\documentclass{amsart}
\usepackage[utf8]{inputenc}
\usepackage{amssymb}
\usepackage{amsthm}
\usepackage{amsmath}
\usepackage{enumerate}
\usepackage{mathtools}
\usepackage{comment}
\usepackage{todonotes}
\usepackage{xcolor}\mathtoolsset{showonlyrefs}
\usepackage{hyperref}

\newtheorem{thm}{Theorem}[section]
 
 \newtheorem{lem}[thm]{Lemma}

 \theoremstyle{definition}

 \theoremstyle{remark}
 \newtheorem{rem}[thm]{Remark}

\title[Brownian path preserving mappings on the Heisenberg group]{Brownian path preserving mappings on the Heisenberg group}

\author{Nikita Evseev}
	\address{Institute for Advanced Study in Mathematics, Harbin Institute of Technology, 150006 Harbin, China
	 and Sobolev Institute of Mathematics, 4 Academic Koptyug avenue,  630090 Novosibirsk, Russia.
	}
	\email{evseev@math.nsc.ru}

\thanks{The work is supported by the Mathematical Center in Akademgorodok under Agreement №~075-15-2022-281 from 05.04.2022. with the Ministry of Science and Higher Education of the Russian Federation.} 

\begin{document}

\maketitle
\begin{abstract}
We study continuous mappings on the Heisenberg group that up to a time change preserve horizontal Brownian motion. 
It is proved that only harmonic morphisms possess this property.
\end{abstract}

\section{Introduction}
Let $f:\mathbb C \to \mathbb C$ be a conformal function and let $B(t)$ be a Brownian motion on $\mathbb C$.
In \cite{L1948} P. L\'evy proved that $f(B(t))$ is again Brownian motion up to a random time change. 
The converse is also true: if $f$ preserve Brownian motion then it is conformal (or anti-conformal).
Then Bernard, Campbell, and Davie in \cite{BCD1979} investigated mappings $f:\mathbb R^n \to \mathbb R^p$ and proved 
that a continuous mapping $f$ preserves Brownian motion iff $f$ is a harmonic morphism. 
They also considered various specific examples. In particular it turned out that $f:\mathbb R^n \to \mathbb R^n$ ($n > 2$) preserve Brownian motion iff it is an affine map.
The last relates to what is known from the works of Fuglede \cite{F1978} and Ishihara \cite{I1979}: 
a map between Riemannian manifolds is a harmonic morphism if and only if it is a horizontally conformal harmonic map.
In \cite{CO1983} Csink and {\O}ksendal solve more general problem: they described $C^2$-mappings that map the path of one diffusion process into the path of another diffusion process. 

In this paper we study continuous mappings between Heisenberg groups $f:\mathbb H^n \to \mathbb H^p$ that preserve horizontal Brownian motion.
Following the approach from \cite{BCD1979} we proved that a continuous mapping $f$ preserves Brownian motion on the Heisenberg group if and only if it is a harmonic morphism.
Close results were obtained by Wang in \cite{W2016}, where images of Brownian motions on the Heisenberg group under conformal maps were studied.
Finally, we should mention that \cite[Theorem 1]{CO1983}
generalizes our Theorem \ref{theorem:main} in case of higher smoothness.

The paper is organized as follows. In Section 2 we provide necessary notions on the Heisenberg group and on horizontal Brownian motion. 
In Section 3 we revise the result on representation of the solution of the Dirichlet problem via Brownian motion.
Then, in Section 4 we introduce and prove the main result.

\section{Preliminaries}
\subsection{The Heisenberg group.}
The Heisenberg group $\mathbb H^n$ is defined as  $\mathbb R^{2n+1} = \mathbb C^{n}\times\mathbb R$ with the group law
$$
(z,t)*(z',t') = \big(z+z', t+t' +2\operatorname{Im}\sum_{j=1}^nz_j\overline{z'}_j \big)
= \big(x+x', y+y', t+t' + 2\sum_{j=1}^n(y_jx'_j-x_jy'_j)\big).
$$
The vector fields 
$$
X_j = \frac{\partial}{\partial x_j} + 2y_j\frac{\partial}{\partial t},\quad
Y_j = \frac{\partial}{\partial y_j} - 2x_j\frac{\partial}{\partial t},\quad
T = \frac{\partial}{\partial t}
$$
are left invariant and form a basis of left invariant vector fields on Heisenberg group $\mathbb H^n$.
The only non-trivial commutator relations are $[X_j,Y_j] = 4T$, $j=1,\dots,n$.
For all $g\in\mathbb H^n$ horizontal distribution $H_g\mathbb H^n=span\{X_1(g),Y_1(g),\dots, X_n(g), Y_n(g) \}$. 
A curve $\gamma:[a,b]\to \mathbb H^n$ is \textit{horizontal} if $\gamma'(t)\in H_{\gamma(t)}\mathbb H^n$ for almost every $t\in[a,b]$.  
Let $\gamma(t) = (\xi_1(t),\zeta_1(t),\dots, \xi_n(t),\zeta_n(t), \eta(t))$, then it can be shown that $\gamma(t)$ is horizontal curve if and only if
$$
\eta'(t) = 2\sum_{j=1}^{n} \big(\xi'_j(t)\zeta_j(t) - \zeta_j'(t)\xi_j(t)\big) \quad \text{ for almost every } t\in[a,b].
$$

A mapping $f:U\to\mathbb H^p$ is called \textit{contact} if $\gamma\circ f$ is horizontal curve for any horizontal curve $\gamma:[a,b]\to U$.
If $f(g)=(u_1(g),v_1(g),\dots, u_n(g),v_n(g),h(g))$ then the contact condition is equivalent to
\begin{align}\label{eq:contact}
X_ih = 2\sum_{j=1}^p v_jX_iu_j - u_jX_iv_j\\ 
Y_ih = 2\sum_{j=1}^p v_jY_iu_j - u_jY_iv_j,
\end{align}
for $i=1,\dots n$.

For any $g=(z,t)\in\mathbb H^n$ define the Kor\'anyi norm 
$$
\rho(g) = (|z|^4+t^2)^{\frac{1}{4}}
$$
and the Kor\'anyi metric $\rho(g_1,g_2) = \rho(g_2^{-1}*g_1)$.

If we are given an absolutely continuous curve $\tilde\gamma(t) = (\xi_1(t),\zeta_1(t),\dots, \xi_n(t),\zeta_n(t)):[a,b]\to\mathbb R^{2n}$,
then by defining 
$$
\eta(t):= \eta(a) + 2\sum_{j=1}^{n}\int_a^t \xi'_j(s)\zeta_j(s) - \zeta_j'(s)\xi_j(s) \ \textrm{d} s 
$$
we obtain a horizontal curve $\gamma=(\tilde\gamma, \eta)$, which is called \textit{the horizontal lift of} $\tilde\gamma$.

\subsection{Horizontal Brownian motion on the Heisenberg group.}
Let $B(t) = \big(B^1_1(t), B^2_1(t), \dots, B^1_n(t), B^2_n(t)\big)$ be a Brownian motion in $\mathbb R^{2n}$ starting at $0$.
Consider the L\'evi area integral
\begin{equation}\label{eq:Levi}
S(t) = 2\sum_{j=1}^{n}\int_0^t B^2_j(s) \ \textrm{d} B^1_j(s) - B^1_j(s) \ \textrm{d} B^2_j(s)
\end{equation}
Then the process $\mathring{W} (t) = \big(B(t), S(t)\big)$, which could be viewed as a horizontal lift of $B(t)$, 
is the solution of a system of stochastic differential equations
\begin{align*}
& \textrm{d} \mathring{W}_{2k-1}(t) = \textrm{d} B^1_k(t), \quad  \textrm{d} \mathring{W}_{2k}(t) = \textrm{d} B^2_k(t),\quad k=1,\dots,n,\\
& \textrm{d} \mathring{W}_{2n+1}(t) = 2\sum_{j=1}^{n} B^2_j(t) \ \textrm{d} B^1_j(t) - B^1_j(t) \ \textrm{d} B^2_j(t).
\end{align*} 
As a consequence $\mathring{W}(t)$ is a Markov process with generator $\frac{1}{2}\Delta_{\mathbb H}$, where 
\[
\Delta_{\mathbb H} = \sum_{j=1}^n X^2_j + Y^2_j
\]

Let $g$ be a given point in $\mathbb H^n$, then the horizontal Brownian motion starting at this point is defined as $W(t) = g*\mathring{W}(t)$.

We will need the following It\^o formula for horizontal Brownian motion, see \cite[lemma 3.2]{BGKNT2020}. 
\begin{lem}[It\^o formula]
\label{lemma:Ito}
Let $f\in C^2(\mathbb H^n; \mathbb R)$ and $W(t)$ be a horizontal Brownian motion in $\mathbb H^n$.
Then 
\begin{equation}\label{eq:Itoformula}
f(W(t)) = f(W(0)) + \int_0^t \nabla_{\mathbb H}f(W(s)) \,\cdot \, \textrm{d} B(s) + \frac{1}{2} \int_0^t \Delta_{\mathbb H}f(W(s))  \ \textrm{d} s.
\end{equation}
\end{lem}

We are going to use the following lemma for 1-dimensional Brownian motion.
\begin{lem}[{\cite[Problem 1, p. 45]{McK1969}} ]\label{lemma:McKean}
If $e:[0,\infty]\to \mathbb R^d$ and $\sigma(t) = \int_0^t|e|^2(s) \ \textrm{d} s$,
then $a(s) = \int_0^{\sigma^{-1}(s)}e(q) \,\cdot\, \textrm{d} B(q)$
is 1-dimensional Brownian motion.
\end{lem}

And we will use the following time change formula for It\^o integrals.
\begin{thm}[{\cite[Theorem 8.5.7, p. 156]{O2003}} ]\label{theorem:Oksendal}
Suppose $c(s,\omega)$ and $a(s,\omega)$ are $s$-continuous almost surely, $a(0,\omega)=0$ a.s., and that $E |a_t|<\infty$.
Let $B(s)$ be a $d$-dimensional Brownian motion and $d$-vector $v(s,\omega)$ be bounded and $s$-continuous.
Define 
$$
\breve B(s) = \int_0^{a(t)} \sqrt{c(s)} \ \textrm{d}  B(s).
$$
Then $\breve B(s)$ is a Brownian motion and
$$
\int_0^{a(t)} v(s)\, \cdot\, \textrm{d}  B(s) = \int_0^{t} v(a(r))\sqrt{a'(r)} \, \cdot\,  \textrm{d} \breve B(r) \quad \text{a. s.},
$$
where $a'(r)$ is the derivative of $a(r)$ w.r.t. r, so that 
$$
a'(r) = \frac{1}{c(a(r))} \quad \text{ for almost all } r, \text{ a. s.}
$$
\end{thm}

\section{Brownian motion and the Dirichlet problem} 
A function $u:\mathbb H^n\to\mathbb R$ is called harmonic if $\Delta_{\mathbb H}u=0$.
In this section we follow \cite{GV1985} and \cite[Theorem 2.1]{G1977} to obtain 
\begin{thm}\label{theorem:Brownian-Dirichlet}
Let $U$ be an open set in $\mathbb H^n$ and $\varphi$ be a bounded continuous function on $\partial U$.
Let $S_U$ be the first exit time from $U$. Define function
\begin{equation}\label{eq:Eu}
u(g) = E(\varphi(W(S_U)) \mid W(0) = g), \quad g\in U.
\end{equation} 
Then $u$ is harmonic in $U$.
Moreover, if $U$ is a regular domain, then function \eqref{eq:Eu} solves the Dirichlet problem
\begin{equation}
\begin{cases}
\Delta_{\mathbb H} u = 0, \quad \text{ in } U,\\
u|_{\partial U} = \varphi \quad \text{ on } \partial U.
\end{cases}
\end{equation}
\end{thm} 

The point $g_0\in\partial U$ is said to be a \textit{regular point} if $P(S_U=0 \mid W(0)=g_0)=1$
and $U$ is \textit{regular open set} if all points of $\partial U$ are regular.

\begin{lem}\label{lemma:domain-ball}
If $u$ is defined by \eqref{eq:Eu}, then for any $g\in U$ and stopping time $s_0\leq S_U$ we have
$$
u(g) = E(u(W(s_0)) \mid W(0) = g).
$$
\end{lem}
\begin{proof}
Let $\mathcal B_{s_0}$ be the $\sigma$-algebra of events previous to $s_0$, then,
by conditioning by $\mathcal B_{s_0}$ we obtain
$$
u(g) = E\big[ E(\varphi(W(S_U)) \mid \mathcal B_{s_0}) \mid W(0) = g\big].
$$
On the other hand the Markov property gives 
$$
E(\varphi(W(S_U)) \mid \mathcal B_{s_0}) = E(\varphi(W(S_U)) \mid W(0) = W(s_0))
$$
So 
$$
u(g) = E\big[ E(\varphi(W(S_U)) \mid W(0) = W(s_0)) \mid W(0) = g\big] 
= E(u(W(s_0)) \mid W(0) = g).
$$
\end{proof}
 
\begin{lem}\label{lemma:harmonic-measure}
Let $B(g_0,\rho_0)\subset\subset U$. Then the law of $W(S_{B(g_0,\rho_0)})$ knowing $W(0)=g_0$ is 
$$
P(W(S_{B(g_0,\rho_0)}) \in \textrm{d} \sigma(g) \mid W(0)=g_0) 
= \frac{2^{n-2}(\Gamma(\frac{1}{n}))^2}{\pi^{n+1}\rho_0^{2n}}\frac{2|z-z_0|^2}{\|\nabla\rho^4\|(g_0^{-1}*g)}\textrm{d} \sigma(g), 
$$
where $g=(z,t)$, $g_0=(z_0,t_0)$,  $\textrm{d} \sigma(g)$ is the euclidean area element on $\partial B(g_0, \rho_0)$, and $\|\nabla\rho^4\|(z,t) = (16|z|^6 + 4t^2)^{\frac{1}{2}}$.
\end{lem}
\begin{proof}
Let $h$ be a bounded continuous function on $\partial B(g_0, \rho_0)$. 
Consider the Dirichlet problem 
\begin{equation}\label{eq:Dirichlet2}
\begin{cases}
\Delta_{\mathbb H} v = 0, \quad \text{ in } B(g_0, \rho_0),\\
v|_{\partial B(g_0, \rho_0)} = h \quad \text{ on } \partial B(g_0, \rho_0).
\end{cases}
\end{equation}
Then there exists a unique solution of \eqref{eq:Dirichlet2}, and the value in the center $g_0$ could be calculated via
\begin{equation}\label{eq:harmonic-measure}
v(g_0) = \int_{\partial B(g_0, \rho_0)} h(g) \ \textrm{d} \mu_{g_0}^{B(g_0,\rho_0)}(g)
\end{equation}
with 
$$
\textrm{d} \mu_{g_0}^{B(g_0,\rho_0)}(g) 
= \frac{2^{n-2}(\Gamma(\frac{1}{n}))^2}{\pi^{n+1}\rho_0^{2n}}\frac{|z-z_0|^2}{\big(4|z-z_0|^6+(t-t_0 -2\operatorname{Im}\sum_{j=1}^nz_j\overline{z^0_j})^2 \big)^{\frac{1}{2}}}\textrm{d} \sigma(g),
$$
see \cite[Theorem 7.2.9]{BLU2007}.

Now, let $v$ be the solution of \eqref{eq:Dirichlet2}. 
Then $v\in C^2(B(g_0,\rho_0))$ and we can apply the It\^o formula (making use $\Delta_{\mathbb H}v=0$):
$$
v(W(s)) = v(W(0)) + \sum_{j=1}^{n} \int_0^s X_jv(W(t)) \ \textrm{d} B^1_j(t) + \int_0^s Y_jv(W(t)) \ \textrm{d} B^2_j(t)
$$
Then almost surely 
\begin{multline*}
h(W(S_{B(g_0,\rho_0)})) = \lim_{s\to S_{B(g_0,\rho_0)}} v(W(s)) \\
=  v(W(0)) + \sum_{j=1}^{n} \int_0^{S_{B(g_0,\rho_0)}} X_jv(W(t)) \ \textrm{d} B^1_j(t) + \int_0^{S_{B(g_0,\rho_0)}} Y_jv(W(t)) \ \textrm{d} B^2_j(t).
\end{multline*}
It follows that 
$$
E(h(W(S_{B(g_0,\rho_0)})) \mid W(0) = g_0) = v(g_0).
$$
Thus, combining the last equation with \eqref{eq:harmonic-measure} and noting that $h$ was arbitrary we get the result.
\end{proof} 

\begin{lem}\label{lemma:mean-value}
Let $u$ be a bounded function such that for any $g_0\in U$, any $\rho_0\leq\varepsilon$ sufficiently small, we have the mean value property
\begin{equation}\label{eq:mean-value}
u(g_0) = \int_{\partial B(g_0, \rho_0)} h(g) \ \textrm{d} \mu_{g_0}^{B(g_0,\rho_0)}(g).
\end{equation}
Then $u$ is $C^{\infty}$ function and satisfies $\Delta_{\mathbb H}u=0$ in $U$.
\end{lem}
\begin{proof}
Let $g_0\in U$ and $\varepsilon$ be small, then by the Taylor formula
\begin{equation}\label{eq:Taylor-mean}
u(g) = u(g_0) + P_2(u,g_0)(g) + \mathcal O((\rho(g^{-1}_0*g))^3).
\end{equation}
Now we place \eqref{eq:Taylor-mean} inside \eqref{eq:mean-value}.
Due to symmetry all first order terms and second order terms with mixed derivatives will give $0$.
So we have
$$
u(g_0) = u(g_0) + \frac{1}{2}\sum_{j=1}^n\big(X^2_ju(g_0) + Y^2_ju(g_0) \big)\int_{\partial B(g_0, \rho_0)} x_1^2 \ \textrm{d} \mu_{g_0}^{B(g_0,\rho_0)}(g) + \mathcal O (\varepsilon^3)
$$
or
$$
\frac{1}{2}\sum_{j=1}^n\big(X^2_ju(g_0) + Y^2_ju(g_0) \big)\cdot\frac{1}{\varepsilon^2}\cdot\int_{\partial B(g_0, \rho_0)} x_1^2 \ \textrm{d} \mu_{g_0}^{B(g_0,\rho_0)}(g) = o(1). 
$$
The integral in the last equation is of order $\varepsilon^2$.
Thus we obtain $\Delta_{\mathbb H}u=0$ in $U$.
\end{proof}

\begin{lem}\label{lemma:semicontinuous}
For $t>0$, the function $g\mapsto P_g(S_U\leq t)$ is lower semicontinuous on $\mathbb H^n$:
$$
\liminf_{g\to g_0} P_g(S_U\leq t) \geq P_{g_0}(S_U\leq t)
$$  
\end{lem}

\begin{lem}\label{lemma:regular}
If $g_0\in\partial U$ is a regular point then 
$$
\lim_{g\to g_0} E(\varphi(W(S_U)) \mid W(0) = g) = \varphi(g_0).
$$
\end{lem}
\begin{proof}
Let $g_0\in\partial U$ be a regular point.
For $r>0$, let $s_r$ be the exit time from $B(g_0,r)$ for $W(t)$.

First we will prove that 
\begin{equation}\label{eq:ouverts}
\lim_{ \substack{g\to g_0 \\ g\in U}}  P_g(S_U<s_r) = 1.
\end{equation} 
For any $g\in B(g_0,r)$ we have $P_g(s_r>0)=1$. 
Moreover, for any $\varepsilon>0$ there exist $\tau>0$ such that for any $g\in B(g_0,\frac{r}{2})$ holds 
$
P_q(s_r<\tau)<\varepsilon.
$ 
Fix $\varepsilon>0$ and let $\tau$ be such that the above is true.
Then we have
\begin{multline*}
P_g(S_U\leq s_r) = P_g(S_U\leq s_r, s_r\geq \tau) + P_g(S_U\leq s_r, s_r< \tau)\\
= P_g(S_U\leq \tau) + P_g(S_U\leq s_r, s_r< \tau) - P_g(S_U\leq \tau, s_r< \tau)\\
\geq P_g(S_U\leq \tau) - P_g(s_r< \tau) \geq P_g(S_U\leq \tau) - \varepsilon.
\end{multline*}
Now making use above inequality and the applying lemma \ref{lemma:semicontinuous} and the regularity of $g_0$ we derive
\begin{multline*}
 \limsup_{ \substack{g\to g_0 \\ g\in U}}  P_g(S_U<s_r) \geq  \liminf_{ \substack{g\to g_0 \\ g\in U}}  P_g(S_U<s_r) \\
\geq \liminf_{ \substack{g\to g_0 \\ g\in U}}  P_g(S_U<\tau) - \varepsilon 
\geq P_{g_0}(S_U<\tau) - \varepsilon = 1 - \varepsilon.
\end{multline*}
Since $\varepsilon>0$ was arbitrary, we obtain \eqref{eq:ouverts}.

For any $\varepsilon>0$ take $r>0$ so that for every $g_1\in B(g_0,r)\cap\partial U$ holds $|\varphi(g) - \varphi(g_0)|<\varepsilon$.
So
\begin{multline}
|E_g(\varphi(W(S_U))) - \varphi(g_0)| \leq E_g(|\varphi(W(S_U)) - \varphi(g_0)| ) \\
 < \varepsilon + E_g(\varphi(W(S_U)) \mid W(S_U)\not\in B(g_0,r)\cap\partial U)\\
 \leq \varepsilon + 2\max_{\partial U}|\varphi|\cdot P_g(W(S_U)\not\in B(g_0,r)\cap\partial U). 
\end{multline}
Thanks \eqref{eq:ouverts} we find a neighbourhood of $g_0$ so that 
$$
P_g\big(W(S_U)\not\in B(g_0,r)\cap\partial U\big) = P_g(S_U<s_r) < \frac{\varepsilon}{2\max\limits_{\partial U}|\varphi|}.
$$
That completes the proof.
\end{proof} 
 
\begin{proof}[Proof of theorem \ref{theorem:Brownian-Dirichlet}]
Lemmas \ref{lemma:domain-ball}, \ref{lemma:harmonic-measure}, and \ref{lemma:mean-value} ensure that function $u$ defined by \eqref{eq:Eu} is harmonic in $U$.
In the case of regular domain by lemma \ref{lemma:regular} $u$ attains boundary values. 
\end{proof} 
 
\section{Brownian path preserving mappings}

Let $U$ be a domain in $\mathbb H^n$.
A continuous mapping $f:U\to\mathbb H^p$ is said to be \textit{Brownian path preserving}
if for each $g_0\in U$ and for each horizontal Brownian motion $W(t)$ defined on $(\Omega,\mathcal F,P)$, started from $g_0$, there exist:

$(A)$ a mapping $\omega\mapsto\sigma_\omega$ on $\Omega$ such that for each $\omega$ $\sigma_\omega(t)$ is a continuous strictly increasing function on $[0,S_U]$
and such that for any $t>0$ the mapping $\omega\mapsto\sigma_\omega(t)$ is measurable on $\{t<S_U\}\subset\Omega$.
It is also required that for each $s$ the random variable $\sigma(s)$ be independent of the process $\{W^{-1}(s)*W(t) : t>s\}$.

$(B)$ a horizontal Brownian motion $W'(t)$ defined on $(\Omega', \mathcal F', P')$ in $\mathbb H^p$, started at $0$ such that  

$(C)$ on $(\Omega,\mathcal F,P)\times(\Omega', \mathcal F', P')$ the stochastic process $Z(s) = Z(\omega,\omega',s)$ defined for $s\geq 0$ by
$$
\begin{cases}
f(W(\sigma^{-1}(s))), & s<\sigma(S_U)=\lim_{t\to S_U} \sigma(t),\\
f(W(\sigma(S_U)))*W'(s-\sigma(S_U)), & s\geq\sigma(S_U)
\end{cases}
$$
is horizontal Brownian motion started at $f(g_0)$.

\begin{thm}\label{theorem:main}
Let $U$ be a domain in $\mathbb H^n$ and let $f:U\to \mathbb H^p$ 
be a non-constant continuous mapping. 
Then the following is equivalent:

$(i)$  $f$ is Brownian path preserving mapping;

$(ii)$ $f$ is harmonic morphism.
\end{thm}
\begin{proof}
$(i)\Rightarrow(ii)$. 
Let $B(0,R)$ be a ball in $\mathbb H^p$, let $Q=f^{-1}(B(0,R))$, and let $g_0\in Q$.
Define $U_m=\{g\in U : \rho(g)<m, \rho(g, \mathbb H^n\setminus U)>\frac{1}{m}\}$, $Q_m=Q\cap U_m$.
Let $S_U$ be exit time from $U$ and $s_m$ be exit time from $U_m$.
Let $\psi$ be exit time of $Z$ from $B(0,R)$, then
$\theta:=\min\{\psi,\sigma(S_U)\}$ and $\theta_m:=\min\{\psi,\sigma(s_m)\}$ are stopping times.
Consider a harmonic function $u:\mathbb H^p\to\mathbb R$.
Then by theorem \ref{theorem:Brownian-Dirichlet} and lemma \ref{lemma:domain-ball} we have
$$
u\circ f(g_0) = u(f(g_0)) = E_{f(g_0)}\big(u(Z(\psi))\big) = E_{f(g_0)}\big(u(Z(\theta))\big).
$$
Then by the Lebesgue theorem 
\begin{multline*}
E_{f(g_0)}\big(u(Z(\theta))\big) = \lim_{m\to \infty}E_{f(g_0)}\big(u(Z(\theta_m))\big) \\
= \lim_{m\to \infty}E_{g_0}\big(u\circ f(W(\sigma^{-1}(\theta_m)))\big)\\
= \lim_{m\to \infty}E_{g_0}\big(u\circ f(W(\min\{\sigma^{-1}(\psi),s_m\}))\big).
\end{multline*}
Note that $\min\{\sigma^{-1}(\psi),s_m\}$ is the exit time from $Q_m$.
By theorem \ref{theorem:Brownian-Dirichlet} function $v_m(g) = E_{g}\big(u\circ f(W(\min\{\sigma^{-1}(\psi),s_m\}))\big)$
is harmonic in $Q_m$. Therefore $u\circ f$ is harmonic in $Q$. 
Since $R$ is arbitrary $u\circ f$ is harmonic on $U$, meaning that $f$ is a harmonic morphism.

$(ii)\Rightarrow(i)$. Let $f=(f_1,f_2,\dots,f_{2p+1}):U\to\mathbb H^p$ be a harmonic morphism, then the following holds true

\begin{align}
&\Delta_{\mathbb H}f_i=0, &\text{ for } i=1,\dots, 2p+1;\label{eq:hm1}\\
&\langle \nabla_{\mathbb H} f_i, \nabla_{\mathbb H} f_j\rangle = h(g)\cdot \delta_{i,j}, &\text{ for } i,j=1,\dots, 2p;\label{eq:hm2}\\
&f \text{ is a contact mapping}\label{eq:hm3}.
\end{align}
Define 
$$
\sigma(t) = \int_0^t|\nabla_{\mathbb H} f_1|^2(W(s)) \ \textrm{d} s, \quad 0\leq t\leq S_U.
$$
This $\sigma$ satisfies condition $(A)$. 
Let $U_m$ and $s_m$ be as in the previous part of the proof, and let 
$$
\sigma_m(t)=
\begin{cases}
\sigma(t), & t\leq s_m;\\
\sigma(s_m)+ t - s_m, & t>s_m. 
\end{cases}
$$
With $W'$ as in $(B)$ define a process $Z^m(s) = Z^m(\omega,\omega',s)$ by
$$
Z^m(s)=
\begin{cases}
Z(s), & s< \sigma(s_m);\\
f(W(s_m))*W'(s-\sigma(s_m)), & s\geq\sigma(s_m), 
\end{cases}
$$
where $Z(s)$ as in $(C)$. 
Then almost surely $Z^m$ is continuous for $s>0$, and $Z^m(s)\to Z(s)$ when $m\to\infty$ almost surely for each $s$.
We will prove that $Z^m(s)$ is a horizontal Brownian motion on $\mathbb H^p$, which will imply so is $Z(s)$. 

Fix $m$.
First we justify that $Z^m_j(s)$, $j=1,\dots,2p$ are 1-dimensional Brownian motions. 

By the It\^o formula (Lemma \ref{lemma:Ito})
$$
Z^m_1(\sigma_m(t)) =
\begin{cases}
f_1(W(t)) = f_1(W(0)) + \int_0^t \nabla_{\mathbb H} f_1(W(s)) \, \cdot\, \textrm{d} B(s), & t<s_m;\\
f_1(W(s_m)) + B'_1(\sigma_m(t) - \sigma(s_m)), & t>s_m,
\end{cases}
$$

and then 
$$
Z^m_1(s) =
\begin{cases}
f_1(W(0)) + \int_0^{\sigma_m^{-1}(s)} \nabla_{\mathbb H} f_1(W(q)) \, \cdot\, \textrm{d} B(q), & \sigma^{-1}_m(s)<s_m;\\
f_1(W(s_m)) + \tilde B^1_1(s-\sigma(s_m)), & \sigma^{-1}_m(s)>s_m.
\end{cases}
$$
Now we redefine the initial Brownian motion $W$ changing its first coordinate (and, consequently the last one ) after time $s_m$:
$\hat W(t) = W(t)$ when $t\leq s_m$ and $\hat B^1_1(t) = \tilde B^1_1(t-s_m)$ for $t> s_m$. Note that $\hat W$ is defined on the product $\Omega\times\Omega'$.
Then 
$$
Z^m_1(s) = f_1(W(0)) + \int_0^{\sigma_m^{-1}(s)} \nabla_{\mathbb H} f_1(W(q)) \, \cdot\, \textrm{d} \hat B(q) \quad \text{ when } \sigma^{-1}_m(s)<s_m.
$$
For $s\geq \sigma(s_m)$ it holds $s=\sigma^{-1}_m(s) + \sigma(s_m) - s_m$, and
\begin{multline*}
Z^m_1(s) = f_1(W(s_m)) + \tilde B^1_1(\sigma^{-1}_m(s) -s_m)\\
 = f_1(W(s_m)) + \hat B^1_1(\sigma^{-1}_m(s)) - \hat B^1_1(s_m) = f_1(W(s_m)) + \int_{s_m}^{\sigma_m^{-1}(s)} \ \textrm{d}  \hat B^1_1(q).
\end{multline*}
It follows 
$$
Z^m_1(s) = f_1(W(0)) + \int_0^{\sigma_m^{-1}(s)} e(q) \, \cdot\, \textrm{d} \hat B(q),
$$
where
$$
e(q) = 
\begin{cases}
\nabla_{\mathbb H} f_1(W(q)), & \text{ if } q< s_m;\\
e_1,  & \text{ if } q\geq s_m.
\end{cases}
$$
So, due to lemma \ref{lemma:McKean} $Z^m_1(s)$  is 1-dimensional Brownian motion. 
In the same manner we prove this fact for other horizontal coordinates $Z^m_j(s)$, $j=2,\dots,2p$.

Now we should prove that $Z^m_{2p+1}(s)$ is the L\'evi area integral \eqref{eq:Levi} of horizontal components.

So, with theorem \ref{theorem:Oksendal}  we have 
$$
\int_0^{\sigma_m^{-1}(s)} e_j(q) \, \cdot\, \textrm{d} \hat B(q) = \int_0^{s} e_j(r)\frac{1}{|e|(\sigma_m^{-1}(r))} \, \cdot\,  \textrm{d} \breve B(r).
$$
Therefore 
\begin{equation}\label{eq:dZj}
\textrm{d} Z^m_j(s) = e_j(s)\frac{1}{|e|(\sigma_m^{-1}(s))} \, \cdot\,  \textrm{d} \breve B(s).
\end{equation}

For the vertical component ( $j=2p+1$ ) we apply It\^o formula (taking into account \eqref{eq:hm1}) and then contact condition \eqref{eq:contact}, in the case $s\leq\sigma(s_m)$:
\begin{multline*}
Z^m_{2p+1}(\sigma_m(t)) = f_{2p+1}(W(t)) = f_{2p+1}(W(0)) + \int_0^t \nabla_{\mathbb H} f_{2p+1}(W(s)) \, \cdot\, \textrm{d} B(s) \\
= f_{2p+1}(W(0)) + \sum_{i=1}^{p} \int_0^t 2\sum_{j=1}^{p} 
(f_{2j}X_if_{2j-1} - f_{2j-1}X_if_{2j}) \ \textrm{d} B^1_i(s)\\
+ (f_{2j}Y_if_{2j-1} - f_{2j-1}Y_if_{2j}) \ \textrm{d} B^2_i(s)\\
= f_{2p+1}(W(0)) + 2\sum_{j=1}^{p} \int_0^t f_{2j}(W(s))\nabla_{\mathbb H}f_{2j-1}(W(s))\, \cdot\, \textrm{d} B(s)\\
- f_{2j-1}(W(s))\nabla_{\mathbb H}f_{2j}(W(s))\, \cdot\, \textrm{d} B(s).
\end{multline*}
So we have
\begin{multline}\label{eq:Z2p+1}
Z^m_{2p+1}(s) = f_{2p+1}(W(0)) + 2\sum_{j=1}^{p} \int_0^{\sigma^{-1}_m(s)} f_{2j}(W(q))\nabla_{\mathbb H}f_{2j-1}(W(q))\, \cdot\, \textrm{d} \hat B(q)\\
- f_{2j-1}(W(q))\nabla_{\mathbb H}f_{2j}(W(q))\, \cdot\, \textrm{d} \hat B(q).
\end{multline}

For $s\geq \sigma(s_m)$ it holds $s=\sigma^{-1}_m(s) + \sigma(s_m) - s_m$, and $\tilde S(\sigma^{-1}_m(s) -s_m) = \hat S(\sigma^{-1}_m(s)) - \hat S(s_m)$,
so
\begin{multline*}
Z^m_{2p+1}(s) = f_{2p+1}(W(s_m)) + \tilde S(\sigma^{-1}_m(s) -s_m)\\
 + 2\sum_{j=1}^{p}f_{2j}(W(s_m))\tilde B^1_j(\sigma^{-1}_m(s) -s_m) - f_{2j-1}(W(s_m))\tilde B^2_j(\sigma^{-1}_m(s) -s_m)  \\
 = f_{2p+1}(W(s_m))  +  \int_{s_m}^{\sigma_m^{-1}(s)} \ \textrm{d}  \hat S(q)\\
+ 2\sum_{j=1}^{p} \int_{s_m}^{\sigma_m^{-1}(s)} f_{2j}(W(s_m))  \ \textrm{d}  \hat B^1_j(q) - \int_{s_m}^{\sigma_m^{-1}(s)} f_{2j-1}(W(s_m))  \ \textrm{d}  \hat B^2_j(q)\\
 = f_{2p+1}(W(s_m))  +  2\sum_{j=1}^{p}\int_{s_m}^{\sigma_m^{-1}(s)} \hat B^2_j(q) \ \textrm{d} \hat B^1_j(q) -\hat B^1_j(q) \ \textrm{d} \hat B^2_j(q)\\
+ 2\sum_{j=1}^{p} \int_{s_m}^{\sigma_m^{-1}(s)} f_{2j}(W(s_m))  \ \textrm{d}  \hat B^1_j(q) - \int_{s_m}^{\sigma_m^{-1}(s)} f_{2j-1}(W(s_m))  \ \textrm{d}  \hat B^2_j(q)\\
 = f_{2p+1}(W(s_m))  
+ 2\sum_{j=1}^{p} \int_{s_m}^{\sigma_m^{-1}(s)} f_{2j}(W(s_m)) +\hat B^2_j(q) \ \textrm{d}  \hat B^1_j(q)\\ -  \big( f_{2j-1}(W(s_m)) + \hat B^1_j(q)\big)  \ \textrm{d}  \hat B^2_j(q).
\end{multline*}
From the last and \eqref{eq:Z2p+1} we derive
$$
Z^m_{2p+1}(s) = f_{2p+1}(W(0)) + \int_0^{\sigma_m^{-1}(s)} e_{2p+1}(q) \, \cdot\, \textrm{d} \hat B(q),
$$
where
$$
e_{2p+1}(q) = 
\begin{cases}
2\sum_{j=1}^{p} f_{2j}(W(q))\nabla_{\mathbb H}f_{2j-1}(W(q)) - f_{2j-1}(W(q))\nabla_{\mathbb H}f_{2j}(W(q)), & \text{ if } q< s_m;\\
\begin{bmatrix}
           f_{2}(W(s_m)) +\hat B^2_1(q) \\
           -f_{1}(W(s_m)) - \hat B^1_1(q) \\
           \vdots \\
           f_{2p-1}(W(s_m)) + \hat B^1_p(q)
         \end{bmatrix},  & \text{ if } q\geq s_m.
\end{cases}
$$
Again, by theorem \ref{theorem:Oksendal}  we have 
$$
\int_0^{\sigma_m^{-1}(s)} e_{2p+1}(q) \, \cdot\, \textrm{d} \hat B(q) = \int_0^{s} e_{2p+1}(r)\frac{1}{|e|(\sigma_m^{-1}(r))} \, \cdot\,  \textrm{d} \breve B(r).
$$
Therefore, taking into account  \eqref{eq:dZj}
$$
\textrm{d} Z^m_{2p+1}(s) = e_{2p+1}(s)\frac{1}{|e|(\sigma_m^{-1}(s))} \, \cdot\,  \textrm{d} \breve B(s)
=  2\sum_{j=1}^{p} Z^m_{2j} \ \textrm{d} Z^m_{2j-1} - Z^m_{2j-1} \ \textrm{d} Z^m_{2j}.
$$
Thus we have proved that $Z^m(s) = \big(Z^m_1(s), Z^m_2(s),\dots, Z^m_{2p+1}(s)\big)$ is a horizontal Brownian motion.
\end{proof} 

\begin{thm}
Let $U$ be a domain in $\mathbb H^n$ and let $f:U\to\mathbb H^n$ be a Brownian path preserving mapping.
Then $f=\pi_b\circ\varphi_A\circ\delta_{\alpha}|_{U}$, i. e. $f$ is the restriction on $U$ of the composition of translation, rotation, and dilatation.  
\end{thm}
\begin{proof}
Let $f:U\to\mathbb H^n$ is a Brownian path preserving mapping. 
Due to theorem \ref{theorem:main} $f$ is a harmonic morphism, so by  \eqref{eq:hm1} and \eqref{eq:hm2} we have
\[
\|D_{H}f(x)\|^{2n+2} = |J(x,f)|,
\]
where $D_{H}f$ and $J(\cdot,f)$ are the formal horizontal differential and the formal Jacobian of $f$.
The last equation means that distortion coefficient of $f$ equals $1$.
Then, by \cite[Theorem 12]{V1999} mapping $f$ is constant or the restriction of some M\"obius transform to $U$.
It remains to note that  translation, rotation, and dilatation are harmonic morphisms, but inversion is not.
\end{proof}

\begin{rem}
In the case $U\subset \mathbb H^n$ and $p<n$ no nontrivial map $f:U\to\mathbb H^p$ is contact.
Therefore there are no harmonic morphisms in this situation.
\end{rem}
    
\bibliographystyle{plain}
\bibliography{bibliography}
\end{document}